\documentclass[10pt,oneside]{amsart}

\usepackage{geometry}

\usepackage{amsmath,amssymb,amsthm}
\usepackage{mathtools}
\usepackage{bbm}
\usepackage{enumerate}

\usepackage{xcolor}
\definecolor{myGreen}{rgb}{0.18039216 0.49803922 0.09411765}
\definecolor{darkred}{rgb}{0.65, 0.0, 0.0}

\usepackage{algorithm}
\usepackage{algorithmic}

\usepackage[colorlinks,linkcolor=myGreen,citecolor=myGreen,urlcolor=darkred]{hyperref}
\usepackage[nameinlink,noabbrev]{cleveref} 
\usepackage{thmtools} 

\declaretheorem[name=Theorem,numberwithin=section]{theorem}
\declaretheorem[name=Lemma,sibling=theorem]{lemma}
\declaretheorem[name=Proposition,sibling=theorem]{proposition}
\declaretheorem[name=Corollary,sibling=theorem]{corollary}
\declaretheorem[name=Definition,sibling=theorem,style=definition]{definition}
\declaretheorem[name=Remark,sibling=theorem,style=remark]{remark}
\declaretheorem[name=Example,sibling=theorem,style=remark]{example}
\crefname{algorithm}{Algorithm}{Algorithms}
\Crefname{algorithm}{Algorithm}{Algorithms}

\newcommand{\R}{\mathbb{R}}

\DeclareMathOperator{\GL}{GL}

\newcommand\rank{\mathsf{rk}}

\title[An Efficient Algorithm for Path Recovery from Signature Tensors]{An Efficient Algorithm for Path Recovery \\from Signature Tensors}

\author[L. Schmitz]{Leonard Schmitz}
\address{
TU Berlin, Algebraic and Geometric Methods in Data Analysis}
\email{lschmitz@math.tu-berlin.de}
\thanks{The author acknowledges funding by the Deutsche
Forschungsgemeinschaft (DFG, German Research Foundation) – CRC/TRR 388
``Rough Analysis, Stochastic Dynamics and Related Fields'' – Project A04, 
516748464.}

\date{}

\begin{document}
\begin{abstract}
We present a new algorithm for recovering paths from their third-order signature tensors, an inverse problem in rough analysis. Our algorithm provides the exact solution to this recovery problem and improves upon current approaches by an order of magnitude. It relies on generalized normal forms and stabilizers of group actions via matrix-tensor congruence. We apply randomized transformation techniques that avoid the task of solving nonlinear polynomial systems associated to degenerate paths, and accompany our methods with an efficient implementation in the computer algebra system \texttt{OSCAR}. 
\end{abstract}
\maketitle



\section{Introduction}

An important inverse problem in rough analysis is the task of recovering a path from Chen’s
iterated-integrals signature.
At third truncation level and \emph{dimension} $d\geq 2$, the signature inversion can be formalized through the group orbit of a special $d\times d\times d$ \emph{core tensor} $C$.    
For any invertible matrix $A\in\GL_d$, the \emph{matrix-tensor congruence action} on $C$ yields a tensor $G:=A*C$ with entries
\begin{equation}\label{eq:learningTask}G_{i,j,k}=\sum_{\alpha=1}^d\sum_{\beta=1}^d\sum_{\gamma=1}^dC_{\alpha,\beta,\gamma}A_{i,\alpha}A_{j,\beta}A_{k,\gamma}\end{equation}
for $1\leq i,j,k\leq d$. Hence, the tensor $G$ admits a (symmetric) Tucker decomposition; see \cite[Section 4]{kolda2009tensor}. 
In the context of rough analysis, the  matrix $A$ encodes a piecewise linear path with $d$ segments, and $G$ its third level signature. Throughout this work, the core tensor $C$ is fixed and represents the signature of the \emph{canonical axis path} in $\R^d$. 
We consider $G$ as an element in the group orbit ${\GL_d}*C$ under the action \eqref{eq:learningTask}. 
The \emph{recovery task} is to find $A\in\GL_d$ such that  $G=A*C$, thereby recovering the path from its signature. 
In \cite[Theorem 6.2]{bib:PSS2019} it is shown that this problem has a unique (real) solution.
Our main result is an efficient algorithm (\Cref{alg:linear_learning}) that provably solves the recovery task.
\begin{theorem}\label{thm:mainResult} For every $G\in {\GL_d}*C$, 
    \Cref{alg:linear_learning} computes $A\in\GL_d$ such that $G=A*C$. 
     In expectation, it requires $\mathcal{O}(d^4)$ elementary operations. 
\end{theorem}

Our main strategy is to compute Gau{\ss} transformations until we reach the core tensor, which acts as our tensor analogous to congruence normal forms of matrices. By inverting all of our transformations, we recover the original path.
We implement our methods in \texttt{OSCAR} \cite{OSCAR-book}, based on exact arithmetic, and report on wall clock times; see \Cref{table:maintableTimings}.

In particular, we outperform previous computations \cite[Table 1]{bib:PSS2019} or \cite[Section~6]{bib:AFS2019} based on Gr\"obner bases \cite{buchberger1965algorithmus}, where $G=A*C$ is considered as a non-linear system of $d^3$ inhomogeneous polynomials in $d^2$ variables, and Gr\"obner bases suffer from a curse of dimensionality, e.g., \cite{dube1990structure}. 

\begin{table}[ht]
\caption{Averaged path recovery timings (in seconds, rounded down, on a MacBook Pro M5 24GB)}
    \centering
    \begin{small}
    $\begin{array}{|r|cccccc||ccccccc|}
    \hline
   d & 2& 3& 4 & 5 & 6&7&10&20&30&40&50&60&70\\\hline
   \text{GBF4}&0&0&0& 1& 35 & \textsc{oom}  & -
        & - & - & - & - & - & - \\
  \textbf{ours}&0&0& 0& 0& 0
&0& 0&0&1&5&16 &29&51\\
   \hline
\end{array}$
\end{small}
\label{table:maintableTimings}
\end{table}

All our comparisons with Gr\"obner bases (GBF4) use the \emph{F4 algorithm} \cite{faugere1999new} provided by \texttt{OSCAR} (in fact \texttt{msolve} \cite{berthomieu2021msolve}). We follow the convention from \cite[Table~1]{bib:PSS2019} and report only on the wall clock times for pure Gr\"obner computations. 
Our implementation accompanying this paper is available at 
\begin{center}
\url{https://github.com/leonardSchmitz/efficient-path-recovery-from-signatures}.
\end{center}

\subsection*{Related work}

Signature inversion is a well-studied problem; see   \cite{lyons2017hyperbolic,lyons2018inverting} for a general overview. 
In \cite{fermanian2024insertion} an \emph{insertion method} is proposed, comparing increasing truncation levels and minimizing their differences. Recently, \cite{vaucher2025lie} introduced a \emph{Lie-adaptive method} that fixes an exact path length and thus changes the optimization problem, while \cite{rauscher2025shortest} explores optimal control approaches and connections to sub-Riemannian geodesics. 

Note that all those works rely on, or are presented for, low-dimensional setups with adjusted truncation level, and therefore lie in a regime different from the approach followed by this work. 

In particular, optimization methods suffer from the relatively high number of unknowns and degrees of the objective function. For example, \cite[Table 2]{bib:PSS2019} reports a $85\%$ \emph{success rate} in dimension $d=10$ for numerically feasible subclasses of piecewise linear paths, combining the Broyden-Fletcher-Goldfarb-Shanno algorithm with a trust region Newton method. 
Here, our method requires $0.006$ seconds on average and has a $100\%$ success rate by design. 

Certain recovery guarantees for signature barycenters appear in \cite{amendola2025learning}, relying on piecewise linear path with few segments and second truncation level. 

Neural-network-based inversion strategies have been proposed in, e.g., \cite{NEURIPS2019_d2cdf047}; however, these approaches incur high computational costs and provide limited theoretical guarantees.


\section{Signature tensors and congruence orbits}\label{sec:sigTenCongr}
We recall the notion of Chen's iterated-integrals signature and its associated group orbits with respect to the matrix-tensor congruence action. 
A \emph{path} 
$
X : [0,1] \to \R^d
$
is a (smooth enough) continuous curve such that the following \emph{iterated integrals} 
\begin{align}
  \label{eq:integrals}
  \sigma(X)_{i,j,k}:=\int_0^1\int_0^{t_3}\int_0^{t_2} \dot X_{i}(t_1) \dot X_j(t_2) \dot X_k(t_3)\,\mathrm dt_1 \mathrm dt_2 \mathrm dt_3
\end{align}
are defined for all 
$1\leq i,j,k\leq d$. In this case, \eqref{eq:integrals} defines a real-valued $d\times d\times d$ tensor $\sigma(X)$, called the \emph{(third level) signature} of $X$. For the definition of higher level signatures we refer to \cite[Definition 7.2]{Friz_Victoir_2010}. 

We recall the following special path that serves as our \emph{dictionary} for path recovery according to  \cite[Equation (8)]{bib:PSS2019}. 

\begin{definition}
 Let $
  \psi:[0,1]\rightarrow\R^d$ denote the axis path of order $d$, defined as the unique piecewise linear path with $d$ segments whose graph interpolates the $d+1$ equidistant support points $$
  \left(\frac id, \sqrt[3]6\sum_{j=1}^ie_j\right) \in [0,1] \times \R^d$$ in consecutive order $0 \leq i \leq d$. Here, $e_i$ denotes the standard basis of $\R^d$.  
\end{definition}

Using this special path, we define our $d\times d\times d$ \emph{core tensor} $C:=\sigma(\psi)$ through the third signature of $\psi$. We recall the closed-form expression from \cite[Example 2.1]{bib:AFS2019} for all its indices $1\leq i,j,k\leq d$, 
\begin{equation}\label{eq:coretensor}
C_{i,j,k}=\begin{cases}1&i=j=k,\\
3&i<j=k\;\text{ or }\;i=j<k,\\
6&i<j<k,\\
0&\text{elsewhere.}
\end{cases}\end{equation}
 The expression \eqref{eq:coretensor} can serve as an alternative definition. Without loss of generality, we have scaled the canonical axis path, and thus normalized our core tensor (with constant $1$ on its diagonal), allowing us to perform all our computations with rational numbers; see \cite[Example 7.13]{amendola2025learning} for an illustration. Note further, that $C$ is \emph{upper triangular} (\cite[eq (9) and p. 13]{bib:PSS2019} or \cite[p. 25]{bib:AFS2019}),  meaning that its support lies on the simplex $1\leq i\leq j\leq k\leq d$. This is an important observation and essential for our recovery algorithm. 
\begin{example}\label{ex:coretensor_d4}
   Here and throughout, we illustrate all $3$-tensors $G\in\R^{d\times d\times d}$ by the (first mode) tensor-to-matrix folding 
   \begin{equation*}\left[
    \begin{array}{ccccccccccc}
    G_{1,1,1} & \dots & G_{1,d,1} & \vrule &  &  &  & \vrule & G_{1,1,d} & \dots & G_{1,d,d} \\
    \vdots & \ddots & \vdots  & \vrule &  & \dots&   & \vrule &\vdots & \ddots & \vdots   \\
    G_{d,1,1} & \dots & G_{d,d,1}  & \vrule &  &  &  & \vrule & G_{d,1,d} & \dots & G_{d,d,d}  \end{array}\right]
\end{equation*}

\noindent
   according to \cite[Section 2.4]{kolda2009tensor}, e.g., 
\begin{equation*}\left[
    \begin{array}{ccccccccccccccccccc}1 & 0 & 0 & 0 & \vrule & 3 & 3 & 0 & 0 & \vrule & 3 & 6 & 3 & 0 & \vrule & 3 & 6 & 6 & 3 \\0 & 0 & 0 & 0 & \vrule & 0 & 1 & 0 & 0 & \vrule & 0 & 3 & 3 & 0 & \vrule & 0 & 3 & 6 & 3 \\0 & 0 & 0 & 0 & \vrule & 0 & 0 & 0 & 0 & \vrule & 0 & 0 & 1 & 0 & \vrule & 0 & 0 & 3 &3 \\0 & 0 & 0 & 0 & \vrule & 0 & 0 & 0 & 0 & \vrule & 0 & 0 & 0 & 0 & \vrule & 0 & 0 & 0 & 1\end{array}\right]
\end{equation*}

\noindent
illustrates the core tensor $C$ with entries \eqref{eq:coretensor} in dimension $d=4$.  
\end{example}

The matrix-tensor congruence is a group action of $\GL_d$ on tensors over $\R^d$. For $A\in\GL_d$ and $C\in\R^{d\times d\times d}$ it is defined by $A*C:=A\times_1A\times_2A\times_3C$ where $\times_i$ denotes the $i$-mode matrix-tensor multiplication \cite[Section 2.5]{kolda2009tensor}. The explicit formula for the entries of the resulting tensor is given in \eqref{eq:learningTask}. 
Since integration is linear, it is easy to see that the signature map is equivariant; see also \cite[Lemma 2.1]{bib:PSS2019}. 
\begin{lemma}\label{lem:equivariance_signature}
For all $X:[0,1]\rightarrow\R^d$ and $A\in\R^{d\times d}$,
\begin{equation}
\sigma(A* X)=A*\sigma(X),
\end{equation}
where $A* X:[0,1]\rightarrow\R^d,t\mapsto A\,X(t)$ is the \emph{linearly transformed axis path}, and $A$ acts on $\sigma(X)$ via the matrix-tensor congruence operation.
\end{lemma}

Clearly, we can write every piecewise linear path $X:[0,1]\rightarrow \R^d$ with $d$ segments as a linearly transformed axis path $X=A*\psi$ for a suitable matrix $A\in\R^{d\times d}$. In fact, the (cumulative) columns of $A$ are precisely those points in $X$ that are allowed to be nondifferentiable. Inspired by \cite{futorny2008tridiagonal} let $\oplus$ denote the \emph{direct sum of matrices}. 
\par
The \emph{orbit} of $C$ under the group $\mathrm{I}_s\oplus\GL_{d-s}:=\{\mathrm{I}_s\oplus A\mid A\in\GL_{d-s}\}$ is the set of elements to which $C$ can be transformed by the group action. We denote it by $(\mathrm{I}_s\oplus\GL_{d-s})*C$. 
\par
In \cite[Theorem 6.2]{bib:PSS2019} it is shown that the \emph{stabilizer} of the core tensor $C$ is \emph{trivial} with respect to the action under $\GL_d$, i.e., if $C=A*C$ for some $A\in\GL_d$ then $A=\mathrm{I}_d$. 

The main strategy is to compute Gau{\ss} transformations $Q\in\GL_d$ and update them at each iteration step $s\in\{1,\dots,d-1\}$. We do this until we  reach $Q*G=C$, which serves as our tensor analogous to congruence normal forms of matrices, \cite[Theorem 1.1]{futorny2008tridiagonal}, \cite[Theorem 1.1]{horn2006canonical}, or \cite[Theorem 2]{lee1996note}. Finally, we invert all Gau{\ss} operations and obtain our recovery $A=Q^{-1}$ with $G=A*C$, using the result of trivial stabilizers mentioned above. 
First we generalize this result on stabilizers, so that we can iteratively use it in our path recovery algorithm. 
\begin{theorem}\label{thm:substab}
    For $G\in{\GL_d}*C$ and $1\leq s \leq d$, the following are equivalent:
    \begin{enumerate}[(i)]
        \item $G_{i,j,k}=C_{i,j,k}$
    for all $1\leq i,j,k\leq d$ with 
    $i=j=k\leq s$ or $k\leq
    \min(s,i-1,j)$ or $j\leq\min(s,i-1,k)$. 
    \item $G\in(\mathrm{I}_s\oplus {\GL_{d-s}})*C$. 
    \end{enumerate}
\end{theorem}
Note that for $s=d$ this implies the original result on trivial stabilizers from \cite[Theorem 6.2]{bib:PSS2019}. In Examples \ref{ex:tensor_running_example_start_s3} and \ref{ex:d4recovery_precont} we illustrate the theorem for $s=1$ and $s=2$, respectively.  
\begin{example}\label{ex:tensor_running_example_start_s3} Consider a given $G\in{\GL_4}*C$ with entries
 \begin{equation*}
\left[\begin{array}{ccccccccccccccccccc}
\mathbf1 & 0 & 0 & 0 & \vrule 
& 6 & 12 & 3 & 9 & 
\vrule & 3 & 9 & 3 & 9 & 
\vrule & 6 & 15 & 3 & 12 \\
\mathbf0 & \mathbf0 & \mathbf0 & \mathbf0 & 
\vrule & \mathbf0 & 8 & 4 & 10 & 
\vrule & \mathbf0 & 7 & 4 & 10 & 
\vrule & \mathbf0 & 10 & 4 & 13 \\
\mathbf0 & \mathbf0 & \mathbf0 & \mathbf0 & 
\vrule & \mathbf0 & 1 & 1 & 1 & 
\vrule & \mathbf0 & 1 & 1 & 1 & 
\vrule & \mathbf0 & 1 & 1 & 1 \\
\mathbf0 & \mathbf0 & \mathbf0 & \mathbf0 & 
\vrule & \mathbf0 & 4 & 4 & 7 & 
\vrule & \mathbf0 & 4 & 4 & 7 & 
\vrule & \mathbf0 & 4 & 4 & 8\end{array}\right]
\end{equation*}
\noindent
we can use \Cref{thm:substab} for $s=1$ and obtain $G\in(\mathrm{I}_1\oplus \GL_3)*C$. We have marked all entries in bold that satisfy the conditions from the first part of the theorem.
\end{example}

For a proof, we state several auxiliary results. We start with the backward direction of \Cref{thm:substab} and consider the following implications. 

\begin{lemma}\label{lem:IsBonC}
    If $G\in(\mathrm{I}_s\oplus\GL_{d-s})*C$, then $G_{i,j,k}=C_{i,j,k}$ for all  $1\leq i,j\leq d$ and $1\leq k\leq s$. 
\end{lemma}
\begin{proof} Let $k\leq s$ and $G:=A*C$ for some $A\in\mathrm{I}_s\oplus\GL_{d-s}$.  If $1\leq i,j\leq s$ then 
\begin{align*}G_{i,j,k}
&=\sum_{\alpha=1}^d\sum_{\beta=1}^d\sum_{\gamma=1}^dC_{\alpha,\beta,\gamma}A_{i,\alpha}A_{j,\beta}A_{k,\gamma}=C_{i,j,k}A_{i,i}A_{j,j}A_{k,k}=C_{i,j,k}
\end{align*}
with the special shape of $A$. If, on the other hand, $i>s$ or $j>s$, then 
\begin{align*}G_{i,j,k}
&=\sum_{\alpha=1}^k\sum_{\beta=1}^kC_{\alpha,\beta, k}A_{i,\alpha}A_{j,\beta}A_{k,k}=0=C_{i,j,k}
\end{align*}
where we also used that the core tensor \eqref{eq:coretensor} is upper triangular. 
\end{proof}

\begin{lemma}\label{lem:IsBonC_part2}
If $G\in(\mathrm{I}_s\oplus\GL_{d-s})*C$ then $G_{ijk}=C_{ijk}$ for all  
    $j\leq\min(s,i-1)$ and $1\leq k\leq d$. 
\end{lemma}

\begin{proof}
    Clearly, we have $i>j$, thus $C_{i,j,k}=0$ for all $1\leq k\leq d$. If $G:=A*C$ with $A\in\mathrm{I}_s\oplus\GL_{d-s}$ we obtain 
    \begin{align*}G_{i,j,k}
&=\sum_{\alpha=1}^j\sum_{\gamma=j}^dC_{\alpha,j,\gamma}A_{i,\alpha}A_{j,j}A_{k,\gamma}=0=C_{i,j,k}
\end{align*}
since $j\leq s$ and $A_{i,\alpha}=0$ for all $1\leq \alpha\leq j$. 
\end{proof}

We follow with an immediate consequence of shuffle relations. 

\begin{lemma}\label{lem:diagonalElemCongruence}
    If $G=A*C$ with $A\in\GL_d$ and $C$ according to \eqref{eq:coretensor}, then 
    \begin{equation*}
        G_{i,i,i} = {\left(A_{i,1}+\dots+A_{i,d}\right)}^3
    \end{equation*}
    for all $1\leq i\leq d$. 
\end{lemma}
\begin{proof}
Let $v\in\R^d$ denote the first signature of the piecewise linear path with columns of $A$ as its segments. The first signature stores the differences from start to end point in every dimension, that is, $v_i=A_{i,1}+\dots+A_{i,d}$ for all $1\leq i\leq d$, see \cite[Page 2]{bib:PSS2019}. With shuffle relations, \cite[Section 1.4]{bib:R1993} or \cite{bib:AFS2019}, the third signature of this path is given by $6v^3_i$, so the claim follows by the normalization of our core tensor $C$. 
\end{proof}

We proceed with the forward  direction of  \Cref{thm:substab}. First, we state it for $s=1$. 
\begin{proposition}\label{prop:thm:substabs1}
    If $G\in{\GL_d}*C$ satisfies 
    $G_{i,j,k}=C_{i,j,k}$ for all $1\leq i,j,k\leq d$ with $i=j=k=1$, 
    $k=1<i$ or $j=1<i$, 
    then $G\in(1\oplus {\GL_{d-1}})*C$. 
\end{proposition}
\begin{proof}
The proof is an adaptation of 
\cite[Theorem 6.2]{bib:PSS2019} so that it requires fewer assumptions on $G=A*C$ with $A\in\GL_d$. With $G_{1,1,1}=C_{1,1,1}=1$ and \Cref{lem:diagonalElemCongruence} we obtain the first row sum, 
$A_{1,1}+\dots+A_{1,d}=1$. 
By rearranging summation we can write the entries of $G$ as   
\begin{align}G_{i,j,k}
=
{\sum_{\alpha}A_{i,\alpha}A_{j,\alpha}A_{k,\alpha}}
&+
{\sum_{\alpha<\beta}3A_{i,\alpha}A_{j,\alpha}A_{k,\beta}\label{eq:PPS2019Thm62}}+
{\sum_{\alpha<\beta}3A_{i,\alpha}A_{j,\beta}A_{k,\beta}}
+
{\sum_{\alpha<\beta<\gamma}6A_{i,\alpha}A_{j,\beta}A_{k,\gamma}}
\end{align}
for every $1\leq i,j,k\leq d$. We recall a factorization of \eqref{eq:PPS2019Thm62} from the proof of 
\cite[Theorem 6.2]{bib:PSS2019} in terms of the standard dot product, that is, 
$G_{i,j,1}
=\langle g_{i,1}, A_{j,\bullet}\rangle$  
with 
\begin{equation}\label{eq:proofStrongerStab1}g_{i,1}(\beta):=A_{i,\beta}A_{1,\beta}
+\sum_{\gamma>\beta}3A_{i,\beta}A_{1,\gamma}
+\sum_{\alpha<\beta}3A_{i,\alpha}A_{1,\beta}
+\sum_{\gamma>\beta>\alpha}6A_{i,\alpha}A_{1,\gamma}
\end{equation}
for $1\leq\beta\leq d$. Since $A$ is invertible and $G_{i,j,1}=C_{i,j,1}=0$ for all $j$ and $i>1$ by assumption, we have $g_{i,1}=0$ for all $i>1$. 
We introduce a new factorization of \eqref{eq:PPS2019Thm62} similar to the above one, 
$G_{i,1,k}=\langle h_{i,1},A_{k,\bullet}\rangle $, 
where 
\begin{equation}\label{eq:proofStrongerStab1_1}
h_{i,1}(\gamma):=A_{i,\gamma}A_{1,\gamma}
+\sum_{\gamma>\alpha}3A_{i,\alpha}A_{1,\alpha}
+\sum_{\gamma>\alpha}3A_{i,\alpha}A_{1,\gamma}
+\sum_{\gamma>\beta>\alpha}6A_{i,\alpha}A_{1,\beta}
 \end{equation}
 for $1\leq \gamma\leq d$. By assumption $G_{i,1,k}=C_{i,1,k}$ for $i>1$ and arbitrary $k$. Again, this implies for all $i>1$ that $h_{i,1}=0$, and in particular
 \begin{equation}\label{eq:proofStrongerStab2}
 0=h_{i,1}(1)=A_{i,1}A_{1,1}.
 \end{equation}
Substituting $A_{1,1}+\dots+A_{1,d}=1$ and \eqref{eq:proofStrongerStab2} in \eqref{eq:proofStrongerStab1} yields 
  \begin{equation*}0=g_{i,1}(1)=A_{i,1}A_{1,1}+\sum_{\gamma=2}^d3A_{i,1}A_{1,\gamma}
  =3A_{i,1}\left(\sum_{\gamma=1}^dA_{1,\gamma}\right)=3A_{i,1}\end{equation*}
  for all $i>1$, 
  that is $A_{2,1}=A_{3,1}=\dots=A_{d,1}=0$. 
  We proceed with the first row of $A$. By inserting $\gamma=2$ in \eqref{eq:proofStrongerStab1_1} we obtain
  \begin{equation}\label{eq:proofStrongerStab3}
  0=h_{i,1}(2)=A_{i,2}A_{1,2}
  \end{equation}
  for all $i>1$. We already showed that 
  \begin{equation*}A=
  \begin{bmatrix}
  A_{1,1}&A_{1,2}&\dots&A_{1,d}\\
  0&A_{2,2}&\dots&A_{2,d}\\
  \vdots&\vdots&\ddots&\vdots\\
  0&A_{d,2}&\dots&A_{d,d}\\
  \end{bmatrix}
  \end{equation*}
 is an upper block-triangular matrix, so the lower right block is invertible. Therefore, there exists $i'>1$ such that $A_{i',2}\not=0$. We use \eqref{eq:proofStrongerStab3} and get $A_{1,2}=0$. 
 \par 
 Assume inductively that $A_{1,2}=\dots =A_{1,n-1}=0$ for $3\leq n\leq d$. We show that $A_{1,n}=0$. For this we insert 
 $\gamma=n$ in \eqref{eq:proofStrongerStab1_1}, use the induction hypothesis, and  obtain
 \begin{equation}\label{eq:proofStrongerStab4}
 0=h_{i,1}(n)=A_{i,n}A_{1,n}+\sum_{\alpha=2}^{n-1}3A_{i,\alpha}A_{1,n}
 \end{equation}
 for all $i>1$. For the sake of contradiction, assume that $A_{1,n}\not=0$. Then, we can multiply \eqref{eq:proofStrongerStab4} by the inverse of $3A_{1,n}$ and obtain 
 \begin{equation*}\begin{bmatrix}A_{2,2}\\\vdots \\A_{d,2}\end{bmatrix}
 +
 \dots
 +
 \begin{bmatrix}A_{2,n-1}\\\vdots \\A_{d,n-1}\end{bmatrix}
 +
 \frac13\begin{bmatrix}A_{2,n}\\\vdots \\A_{d,n}\end{bmatrix}
 =0,\end{equation*}
contradicting that $A$ is invertible. Therefore, we conclude $A_{1,n}=0$, which completes our induction. Hence, we know $A_{1,2}=\dots=A_{1,d}=0$. With $A_{1,1}+\dots+A_{1,d}=1$  we obtain $A_{1,1}=1$.
\end{proof}

We can now prove the main result of this section. 

\begin{proof}[Proof of \Cref{thm:substab}]For the backward direction we assume $G\in(\mathrm{I}_s\oplus\GL_{d-s})*C$ and consider three cases for $1\leq i,j,k\leq d$. If $i=j=k\leq s$ or $k\leq\min(s,i-1,j)$ we use \Cref{lem:IsBonC} and obtain $G_{i,j,k}=C_{i,j,k}$. In the remaining case $j\leq \min(s,i-1,k)$ we use \Cref{lem:IsBonC_part2}. \par
For the forward direction we can use a similar recursion as in the proof of \cite[Theorem 6.2]{bib:PSS2019}. For this consider $G=A*C$ with $A\in\GL_d$ and apply \Cref{prop:thm:substabs1}, that is $A=1\oplus A'$ with $A'\in\GL_{d-1}$. If $s>1$ then let $C'$ denote the core tensor for $d-1$ and set $G':=A'*C'$. Clearly we have 
\begin{equation*}
C'_{i-1,j-1,k-1}=C_{i,j,k}=G_{i,j,k}=G'_{i-1,j-1,k-1}
\end{equation*}
for all $2\leq i,j,k\leq d$ with 
    $i=j=k\leq s$ or $k\leq
    \min(s,i-1,j)$ or $j\leq\min(s,i-1,k)$. After changing indices and setting $s':=s-1$, this translates to 
    \begin{equation*}C'_{i',j',k'}=G'_{i',j',k'}
    \end{equation*}
for all $1\leq i',j',k'\leq d-1$
    such that  $i'=j'=k'\leq s'$ or $k'\leq
    \min(s',i'-1,j')$ or $j'\leq\min(s',i'-1,k')$. 
    \par
    Thus, we can use \Cref{prop:thm:substabs1} for $G'$, $C'$ and $s'$, resulting in $A'\in1\oplus\GL_{d-1}$. This implies $A=\mathrm{I}_2\oplus\GL_{d-2}$ and recursively,  $A\in\mathrm{I}_{s}\oplus\GL_{d-s}$. 
\end{proof}

\section{Lower and diagonal Gau{\ss} operations}\label{sec:learningAlgorithm}

Our Gau{\ss} transformations can be grouped in three types: \emph{upper}, \emph{lower}, and \emph{diagonal}. We also make use of permutation matrices, but this becomes relevant only implicitly. 
In this section we begin with the lower and diagonal transformations. Those transformations are especially easy because they can  be read immediately from the tensor. For every $1\leq s<d$  and $y\in\R^{d-s}$ let 
\begin{equation}\label{eq:def_Lsy}L^{(s,y)}:=\mathrm{I}_d+\sum_{\gamma=s+1}^dy_{\gamma-s}\mathrm{E}_{\gamma, s}\end{equation}
 denote our \emph{lower Gau{\ss} operation}. Here $\mathrm{E}_{i,j}$ with $1\leq i,j\leq d$ denotes the standard basis of $\R^{d\times d}$. Furthermore, for every $1\leq s\leq d$ and $h\in\R\setminus\{0\}$ let \begin{equation}\label{eq:def_Dsy}D^{(s,h)}:=\mathrm{I}_{s-1}\oplus{\sqrt[3]{h}}\,\oplus\,\mathrm{I}_{d-s}\end{equation}
 denote our \emph{diagonal Gau{\ss} operation}.

\begin{example}For dimension $d=4$ and iteration step $s=2$, we get
    \begin{equation*}L^{(2,y)}
    =
    \begin{bmatrix}
    1 & 0 &0 &0\\
    0 & 1 &0 &0\\
    0 & y_1 & 1 & 0 \\
    0 & y_2 & 0 & 1
    \end{bmatrix}\in\GL_4
    \quad
    \text{ and }   \quad
    D^{(2,h)}
    =\begin{bmatrix}
    1 & 0 &0 &0\\
    0 & \sqrt[3]{h} &0 &0\\
    0 & 0 & 1 & 0 \\
    0 & 0 & 0 & 1
    \end{bmatrix}\in\GL_4. 
    \end{equation*} 
\end{example}

Assuming that our tensor satisfies the following nonlinear relations, we can use our lower and diagonal operations to transform it into a sub-group orbit with respect to an increased iteration step index. 

\begin{theorem}\label{thm:LowerTransformations}
    Let $1\leq s<d$.  If $H\in(\mathrm{I}_{s-1}\oplus \GL_{d-s+1})*C$ satisfies 
    \begin{enumerate}[(i)]
     \item\label{thm:LowerTransformations1} $H_{s,s,s}\not=0$,
     \item\label{thm:LowerTransformations2} $H_{i,s,s}=H_{s,i,s}$ for all $s\leq i\leq d$,
     \item\label{thm:LowerTransformations3} $H_{i,s,s}H_{s,j,s}=H_{i,j,s}H_{s,s,s}$ for all $1\leq i\leq d$, $s\leq j\leq d$, and 
     \item\label{thm:LowerTransformations31} $H_{i,s,s}H_{s,s,k}=H_{i,s,k}H_{s,s,s}$ for all $s<i,k\leq d$,  
 \end{enumerate}
then $D^{(s,H_{s,s,s}^{-1})}L^{(s,y)}*H\in(\mathrm{I}_{s}\oplus \GL_{d-s})*C$, 
where $y:=-H_{s,s,s}^{-1}\begin{bmatrix}{H_{s,s+1,s}}&\dots& {H_{s,d,s}}\end{bmatrix}$.
\end{theorem}
 We will soon see that we can transform any tensor from the orbit ${\GL_{d}}*C$ so that the nonlinear relations above hold. First, we illustrate the theorem. 

\begin{example}\label{ex:s1_fromH_to_s2}
Consider the given tensor $H\in{\GL_4}*C$ with folding 
 \begin{equation*}
\left[\begin{array}{ccccccccccccccccccc}-1 & -1 & 1 & -1 & \vrule & 5 & -7 & -8 & -4 & \vrule & 4 & -5 & -7 & -5 & \vrule & 5 & -10 & -8 & -7 \\-1 & -1 & 1 & -1 & \vrule & 5 & 1 & -4 & 6 & \vrule & 4 & 2 & -3 & 5 & \vrule & 5 & 0 & -4 & 6 \\1 & 1 & -1 & 1 & \vrule & -5 & 8 & 9 & 5 & \vrule & -4 & 6 & 8 & 6 & \vrule & -5 & 11 & 9 & 8 \\-1 & -1 & 1 & -1 & \vrule & 5 & -3 & -4 & 3 & \vrule & 4 & -1 & -3 & 2 & \vrule & 5 & -6 & -4 & 1\end{array}\right]
\end{equation*}

\noindent
and observe that the conditions \eqref{thm:LowerTransformations1} to \eqref{thm:LowerTransformations31} of \Cref{thm:LowerTransformations} are satisfied for $s=1$. We perform a diagonal scaling with $h=H^{-1}_{1,1,1}=-1$ and a lower transformation with $y=\begin{bsmallmatrix}-1&
 1&
 -1\end{bsmallmatrix}$ and obtain  $G:=D^{(1,h)}L^{(1,y)}*H$ from  Example \ref{ex:tensor_running_example_start_s3}. With \Cref{thm:substab}, we conclude  $G\in(\mathrm{I}_1\oplus\GL_3)*C$.  
\end{example}

\begin{proof}[Proof of \Cref{thm:LowerTransformations}]
First we apply the lower Gau{\ss} transformation from the theorem. For this denote $L:=L^{(s,y)}$ with entries 
  $y_{i}:=-H_{s,s+i,s}/H_{s,s,s}$ and $J:=L*H$. By assumption we have $H\in(\mathrm{I}_{s-1}\oplus \GL_{d-s+1})*C$. With $L\in\mathrm{I}_{s-1}\oplus \GL_{d-s+1}$ and transitivity we see that $J\in(\mathrm{I}_{s-1}\oplus\GL_{d-s+1})*C$. We compute several entries $J_{i,j,k}$ 
   via a case distinction in $1\leq k\leq d$.  
\begin{enumerate}
  \item\label{thmproof:LowerTransformations_1} If $k<s$ we apply \Cref{lem:IsBonC} for $J$ and $s-1$. Hence $J_{i,j,k}=C_{i,j,k}$ for $1\leq i,j\leq d$. Furthermore $J_{i,j,k}=C_{i,j,k}$ if  
    $j\leq\min(s-1,i-1)$ with \Cref{lem:IsBonC_part2}.
    \item\label{thmproof:LowerTransformations_2} We assume $k=s$ and apply a case distinction in $i$ and $j$. 
    \begin{enumerate}[a.]
        \item\label{thmproof:LowerTransformations_2a}
If $i,j\leq s$ then $J_{i,j,s}=H_{i,j,s}$.
\item\label{thmproof:LowerTransformations_2d} If $i,j>s$ then 
\begin{align*}
    J_{i,j,s}&=H_{s,s,s}L_{i,s}L_{j,s}+H_{s,j,s}L_{i,s}L_{j,j}+H_{i,s,s}L_{i,i}L_{j,s}+H_{i,j,s}L_{i,i}L_{j,j}\\
    &=\frac{H_{s,i,s}H_{s,j,s}}{H_{s,s,s}}-\frac{H_{i,s,s}H_{s,j,s}}{H_{s,s,s}}-\frac{H_{i,s,s}H_{s,j,s}}{H_{s,s,s}}+H_{i,j,s}=0
\end{align*}
using assumption \eqref{thm:LowerTransformations2} and  \eqref{thm:LowerTransformations3}. 
\item\label{thmproof:LowerTransformations_2e} If $i>s$ and $j\leq s$ then 
$J_{i,j,s}=H_{s,j,s}L_{i,s}+H_{i,j,s}L_{i,i}=0$ 
since $H_{s,j,s}=0$ with $j\leq s$ and $H_{i,j,s}$ with $i>s\geq j$. 
\end{enumerate}
In particular, case \eqref{thmproof:LowerTransformations_2a} implies $J_{s,s,s}=H_{s,s,s}$. Furthermore, with cases \eqref{thmproof:LowerTransformations_2d} and \eqref{thmproof:LowerTransformations_2e} we obtain $J_{i,j,k}=0=C_{i,j,k}$ for $k=s\leq \min(i-1,j)$.  Finally we obtain $J_{i,j,k}=0=C_{i,j,k}$ for $j\leq \min(s,i-1)$ with \Cref{lem:IsBonC_part2} and case \eqref{thmproof:LowerTransformations_2e}. 
\item\label{thmproof:LowerTransformations_3} We assume $k>s$. If $j\leq\min(s-1,i-1)$, then $J_{i,j,k}=0$ with \Cref{lem:IsBonC_part2}.
 If $s=j<i$, then 
\begin{equation*}
J_{i,j,k}=H_{s,s,s}L_{i,s}L_{k,s}+H_{i,s,s}L_{k,s}+H_{s,s,k}L_{i,s}+H_{i,s,k}=0
\end{equation*}
 using assumption \eqref{thm:LowerTransformations2}, \eqref{thm:LowerTransformations3} and  \eqref{thm:LowerTransformations31} similarly as in case \eqref{thmproof:LowerTransformations_2d}. 
\end{enumerate}
We perform a last scaling with $D:=D^{(s,H_{s,s,s})}$ and obtain 
$G:=D*J=DL*H$. 
We combine cases \eqref{thmproof:LowerTransformations_1} to \eqref{thmproof:LowerTransformations_3} from above and see that $G$ satisfies all assumptions in  \Cref{thm:substab} for given $s$, thus $G\in\left(\mathrm{I}_{s}\oplus\GL_{d-s}\right)*C$. 
\end{proof}

\section{Upper Gau{\ss} transformations}

It remains to discuss the upper Gau{\ss} transformations, which provide a tensor that satisfies the non-linear relations from  \Cref{thm:LowerTransformations}. Contrary to the above, we cannot read the upper transformation from our tensor, but we have to solve a certain linear system instead.
\par 
In the low-dimensional regime, that is, for iteration steps $s\in\{d-2,d-1\}$, this system does not have a unique solution. In those cases, we use the explicit sequence of congruence transformations, elaborated in \Cref{sec:lowdimCase}. 
\par 
In the remaining cases, we can compute a transformation of the following form: for $1\leq s<d$ and $x\in\R^{d-s}$ let 
    \begin{equation}\label{eq:def_Usx}
U^{(s,x)}:=\left(L^{(s,x)}\right)^\top=\mathrm{I}_d+\sum_{\gamma=s+1}^dx_{\gamma-s}\mathrm{E}_{s,\gamma}
    \end{equation} 
denote our \emph{upper Gau{\ss} operation}. The following observation is the key idea for our upper Gau{\ss} transformation. 

\begin{lemma}\label{lem:technicalStepsUsx}
  For $G\in\R^{d\times d\times d}$ and $U^{(s,x)}$ according to \eqref{eq:def_Usx} let $H:=U^{(s,x)}*G$.
  If $i,j\not=s$ then 
    \begin{equation*}H_{i,j,s}-H_{j,i,s}=G_{i,j,s}-G_{j,i,s}+\sum_{\gamma =s+1}^d\left(G_{i,j,\gamma}-G_{j,i,\gamma}\right)x_{\gamma-s}.\end{equation*} 
\end{lemma}
\begin{proof}
   For every $s$ and $x\in\R^{d-s}$ we have 
    \begin{align*}H_{i,j,s}
    &=\sum_{\alpha=1}^d\sum_{\beta=1}^d\sum_{\gamma=1}^dG_{\alpha,\beta,\gamma}U^{(s,x)}_{i,\alpha}U^{(s,x)}_{j,\beta}U^{(s,x)}_{s,\gamma}
    =\sum_{\gamma=s+1}^dG_{i,j,\gamma}x_{\gamma-s}+G_{i,j,s}
    \end{align*}
    since $i,j\not=s$ implies that only for $\alpha=i$ and $\beta=j$ we have a non-zero summand. 
\end{proof}

Recall that we search for an upper Gau{\ss} operation from a given tensor $G$ to $H$ so that the conditions in \Cref{thm:LowerTransformations} are satisfied. In particular, we want $H_{i,j,s}=H_{j,i,s}$ for every $s<i,j\leq d$ in the iteration step $s$. 
This motivates, together with \Cref{lem:technicalStepsUsx}, the definition of the following linear system $Mx=B$. 

\begin{definition}\label{def:MB}
     For $G\in\R^{d\times d\times d}$ and $1\leq s < d-2$ we define the 
     matrix 
$M\in\R^{(d-s)^2\times (d-s)}$ and the vector $B\in\R^{(d-s)^2}$ entrywise through
 \begin{equation*}M_{(a-1)(d-s)+b,\gamma}:= G_{a+s,b+s,\gamma+s}-G_{b+s,a+s,\gamma+s}\end{equation*}
    and $B_{(a-1)(d-s)+b}:= G_{a+s,b+s,s}-G_{b+s,a+s,s}$
    for $1\leq a,b,\gamma\leq d-s$, respectively. 
\end{definition}

Due to symmetry, we can omit several linear relations. 

\begin{remark}\label{rem:redundantEqLinSyst}
Let $M'\in\R^{\binom{d-s}{2}\times (d-s)}$ and $B'\in\R^{\binom{d-s}{2}}$ be determined by $M$ and $B$ according to \Cref{def:MB}, but where all rows with index $a\geq b$ are omitted. We have $Mx=B$ if and only if $M'x=B'$ for all $x\in\R^{d-s}$.  
\end{remark}

\begin{example}\label{ex:d4recovery}
  Consider the following recovery task: given a tensor $G\in{\GL_d}*C$ in first mode matrix folding
  \begin{small}
    \begin{equation*}\left[
    \begin{array}{ccccccccccccccccccc}-8 & 8 & 4 & 10 & \vrule & 8 & -8 & -4 & -10 & \vrule & 7 & -7 & -4 & -10 & \vrule & 10 & -10 & -4 & -13 \\-4 & -2 & 5 & -7 & \vrule & 4 & 1 & -4 & 6 & \vrule & 2 & 2 & -3 & 5 & \vrule & 5 & 0 & -4 & 6 \\13 & -7 & -10 & -4 & \vrule & -13 & 8 & 9 & 5 & \vrule & -10 & 6 & 8 & 6 & \vrule & -16 & 11 & 9 & 8 \\-8 & 2 & 5 & -4 & \vrule & 8 & -3 & -4 & 3 & \vrule & 5 & -1 & -3 & 2 & \vrule & 11 & -6 & -4 & 1\end{array}
    \right]
    \end{equation*}
    \end{small}

\noindent
 we want to recover $A\in{\GL_d}$ such that $G=A*C$.
    We solve the linear system 
    $Mx=B$,  where $[M\mid B]$ is given by 
    \begin{small}
\begin{equation*}
    \begin{bmatrix}
    G_{2,3,2}-G_{3,2,2}&
    G_{2,3,3}-G_{3,2,3}&
    G_{2,3,4}-G_{3,2,4}&\vrule & G_{2,3,1}-G_{3,2,1}\\
    G_{2,4,2}-G_{4,2,2}&
    G_{2,4,3}-G_{4,2,3}&
    G_{2,4,4}-G_{4,2,4}&\vrule &G_{2,4,1}-G_{4,2,1}\\
    G_{3,4,2}-G_{4,3,2}&
    G_{3,4,3}-G_{4,3,3}&
    G_{3,4,4}-G_{4,3,4}&\vrule &G_{3,4,1}-G_{4,3,1}
    \end{bmatrix}
    =
    \begin{bmatrix}-12 & -9 & -15&\vrule &-12\\
    9 & 6 & 12 &\vrule &9\\
    9 & 9 & 12&\vrule &9\end{bmatrix}
\end{equation*}
\end{small}
    \noindent
 according to \Cref{def:MB} and satisfies $\rank(M\mid B)=\rank(M)=3$.
Note that we have already restricted our system to $3$ from the original $9=(d-s)^2$ linear relations. With out loss of generality, is always possible due to symmetry; see Remark \ref{rem:redundantEqLinSyst} for further details.
We compute the unique solution  $x=\begin{bsmallmatrix}1&
 0&
 0\end{bsmallmatrix}$. The resulting tensor $H:=U^{(1,x)}*G$ is presented in Example \ref{ex:s1_fromH_to_s2}, where it is transformed to $(1\oplus\GL_{3})*C$. 
 \end{example}

We provide a parametrized example where the matrix $M$ from above has full rank for all $d$.

\begin{lemma}\label{lem:exFullRank}
  For  $1\leq s<d-2$ and $G:=(\mathrm{I}_{d}+\mathrm{E}_{d,s})*C$ in the orbit of the core tensor $C$ from \eqref{eq:coretensor}, the matrix $M$ according to \Cref{def:MB} satisfies $\rank(M)=d-s$.
   \end{lemma}
\begin{proof}
 We consider the submatrix  $M''\in\R^{(d-s)\times (d-s)}$ of $M$, with entries
    \begin{equation*}M''_{\beta,\gamma}:=\begin{cases}
    G_{s+1,\beta+s+1,\gamma+s}-G_{\beta+s+1,s+1,\gamma+s} &\beta<d-s\\
     G_{d-1,d,j+s}-G_{d,d-1,\gamma+s}& \beta = d-s.
    \end{cases}\end{equation*}
Note that $M''$ is constructed similarly as in Remark \ref{rem:redundantEqLinSyst}. Here, all except the first $d-s-2$ rows and the last row are omitted.  
We have 
    \begin{equation*}
    M''=\begin{bmatrix}
        -\frac12e_{d-s-1} &\frac12\mathrm{I}_{d-s-1}+\mathrm{U}_{d-s-1}-e_{d-s-1}\mathbbm{1}^\top\\
        0&-\frac12e_{d-s-1}^\top -\frac12e_{d-s-2}^\top
    \end{bmatrix}
    \end{equation*}
    where $\mathrm{U}_{d-s-1}$ denotes the strictly upper triangular matrix with constant one above the diagonal. Therefore, $M''$ is invertible and $M$ has full rank.
\end{proof}

\begin{proposition}\label{prop:genericG_Mfullrank}
   The matrix $M$ according to \Cref{def:MB} has full rank in expectation.  
\end{proposition}
\begin{proof}
For every $G\in{\GL_d}*C$,
    the explicit formula \eqref{eq:learningTask} implies that $G$ is a tensor with polynomial entries in $A\in\GL_d$. Therefore, also $M$ and all its minors are polynomial in $A$. Clearly, one of the $(d-s)\times(d-s)$ minors is not the constant zero polynomial due to \Cref{lem:exFullRank}. The variety of a nonzero polynomial has Lebesgue measure zero \cite[Lemma 1]{okamoto1973distinctness}, so $M$ is invertible for almost every evaluation.
\end{proof}

We now construct an explicit change of coordinates so that the linear system is always solvable. Furthermore, the solution of this system provides the upper Gau{\ss} transformation for  \Cref{thm:LowerTransformations}. 
Let $P^{(s,t)}$ denote the $d\times d$ permutation matrix that represents the cycle $(s,t)$ with $s<t$. For simplicity, we also define $P^{(s,t)}:=\mathrm{I}_d$ if $s=t$. Furthermore, let $\R^*:=\R\setminus\{0\}$. 

\begin{proposition}\label{conj:UpperTransformations}
If $1\leq s <d$ and $G\in\left(\mathrm{I}_{s-1}\oplus \GL_{d-s+1}\right)*C$ then 
     there exist $s\leq t$ and $x\in\R^{d-s}$ such that 
    \begin{equation}\label{eq:UpTrafo}
    H:=U^{(s,x)}P^{(s,t)}*G\in\left(\mathrm{I}_{s-1}\oplus\begin{bmatrix}\R^*&0_{d-s}^\top\\\R^{d-s}&\GL_{d-s}\end{bmatrix}\right)*C.\end{equation}
    Furthermore, $H$ satisfies the conditions \eqref{thm:LowerTransformations1} to \eqref{thm:LowerTransformations31} of \Cref{thm:LowerTransformations}. 
\end{proposition}
\begin{proof}
Let $G=(\mathrm{I}_{s-1}\oplus B)*C$ with $B\in\GL_{d-s+1}$. 
We have $P^{(s,t)}(\mathrm{I}_{s-1}\oplus B)\in\mathrm{I}_{s-1}\oplus\GL_{d-s+1}$ for all $t\geq s$. 
Since $B$ is invertible, we can choose $t\geq s$ such that $(B^{-1})_{1,t-s+1}\not=0$. With this, $((\mathrm{I}_{s-1}\oplus B^{-1})P^{(s,t)})_{s,s}\not=0$ so that   
\begin{equation*}
((\mathrm{I}_{s-1}\oplus B^{-1})P^{(s,t)})^{-1}=P^{(s,t)}(\mathrm{I}_{s-1}\oplus B)=:\mathrm{I}_{s-1}\oplus W
\end{equation*}
satisfies $(W^{-1})_{1,1}\not=0$. 
With 
$x:=\frac1{(W^{-1})_{1,1}}\begin{bmatrix}
        (W^{-1})_{1,2}&\dots&(W^{-1})_{1,d-s}\end{bmatrix}$
   we obtain
    \begin{equation*}
    A:=U^{(s,x)}(\mathrm{I}_{s-1}\oplus W)=
    \mathrm{I}_{s-1}\oplus\begin{bmatrix} \frac1{(W^{-1})_{1,1}}&0&\dots &0\
    \\W_{2,1}&W_{2,2}&\dots&W_{2,d-s}\\
    \vdots&&\ddots&\vdots\\
    W_{d-s,1}&W_{d-s,2}&\dots&W_{d-s,d-s}
    \end{bmatrix}
    \end{equation*}
   which proves the inclusion \eqref{eq:UpTrafo}. Let $h:=1/{(W^{-1})_{1,1}}$. 
    To show condition  \eqref{thm:LowerTransformations1} of \Cref{thm:LowerTransformations} we use the special form of $A$ and obtain  $H_{s,s,s}=A_{s,s}A_{s,s}A_{s,s}=h^3\not=0$. 
  If $i>s$ then 
   \begin{equation}\label{eq:proof_conj:UpperTransformations_1}H_{i,s,s}=\sum_{\alpha=s}^dC_{\alpha, s,s}A_{i,\alpha}h^2=A_{i,s}h^2
   \end{equation}
   and similarly $H_{s,j,s}=A_{j,s}h^2$ 
   for $j>s$. This implies condition \eqref{thm:LowerTransformations2}.  For part \eqref{thm:LowerTransformations3} we first assume $i>s$ and $j>s$, so 
\begin{equation*}
H_{i,j,s}=\sum_{\alpha=s}^d\sum_{\beta=s}^dC_{\alpha,\beta, s}A_{i,\alpha}A_{j,\beta}h=A_{i,s}A_{j,s}h=\frac{H_{i,s,s}H_{s,j,s}}{H_{s,s,s}}
\end{equation*}
   using part \eqref{thm:LowerTransformations1} and \eqref{eq:proof_conj:UpperTransformations_1}. Now we assume $i<s$ and $j>s$. 
   Then $H_{i,s,s}=3h^2$, 
   so   
   $H_{i,j,s}=3A_{j,s}h$ which 
   concludes part \eqref{thm:LowerTransformations3}. For part \eqref{thm:LowerTransformations31} we have 
   \begin{equation*}H_{s,s,k}=\sum_{\gamma=s}^dC_{s,s,\gamma}h^2A_{k,\gamma}\quad\text{ and }\quad H_{i,s,k}=\sum_{\alpha=s}^d\sum_{\gamma=s}^dC_{\alpha,s,\gamma}A_{i,\alpha}hA_{k,\gamma}
   \end{equation*}
   so  $H_{i,s,s}H_{s,s,k}=H_{i,s,k}H_{s,s,s}$ when using $H_{s,s,s}=h^3$ and \eqref{eq:proof_conj:UpperTransformations_1}.  
\end{proof}

\begin{corollary}\label{cor:MxB_solvable}
   The linear system $Mx=B$ with  $M$ and $B$ according to \Cref{def:MB} is  generically solvable.
\end{corollary}
\begin{proof}
We use the solution $x$ that is constructed in the proof of \Cref{conj:UpperTransformations}. Generically, we can use $s=t$ so that $B^{-1}_{1,1}=W^{-1}_{1,1}$. This follows, for example, with the adjugate matrix of $B$ and \cite[Lemma 1]{okamoto1973distinctness}. We have \begin{equation*}H_{i,j,s}=\frac{H_{i,s,s}H_{s,j,s}}{H_{s,s,s}}=\frac{H_{j,s,s}H_{s,i,s}}{H_{s,s,s}}=H_{i,j,s}
\end{equation*} for $H:=U^{(s,x)}*G$. In particular, $x$ solves $Mx=B$ according to \Cref{lem:technicalStepsUsx}. 
\end{proof}

We use \Cref{alg:solve} to compute this transformation in iteration steps $1\leq s<d-2$.

\begin{algorithm}
\begin{algorithmic}
\caption{Transformation matrix $\mathsf{up_{\geq4}}(G,s)$}\label{alg:solve}
\STATE\textbf{Input:} $1\leq s<d-2$ and $G\in(\mathrm{I}_{s-1}\oplus\GL_{d-s+1})*C$  \hfill\text{// $C$ according to \eqref{eq:coretensor}}
       \FOR{$1\leq a,b,\gamma\leq d-s$}{
    \STATE$M_{(a-1)(d-s)+b,\gamma}\gets G_{a+s,b+s,\gamma+s}-G_{b+s,a+s,\gamma+s}$ \hfill\text{// $M\in\R^{(d-s)^2\times (d-s)}$}
      \STATE$B_{(a-1)(d-s)+b}\gets G_{a+s,b+s,s}-G_{b+s,a+s,s}$\hfill\text{// $B\in\R^{(d-s)^2}$}
   }
   \ENDFOR
    \IF{$\rank(M)=\rank(M\mid B)=d-s$}{
   \STATE Solve $Mx=B$ for $x\in\R^{d-s}$\;
     \RETURN $U^{(s,x)}$\;
     }
     \ELSE{
   \RETURN $\mathsf{up_{\geq 4}}(W*G,s)W\text{ for random }W\in\mathrm{I}_{s-1}\oplus\GL_{d-s+1}$\;
     }
     \ENDIF
\end{algorithmic}
\end{algorithm}

The algorithm creates a matrix $M\in\R^{(d-s)^2\times (d-s)}$ and a vector $B\in\R^{(d-s)^2}$ that contain certain linear combinations of the entries of our input tensor $G\in\R^{d\times d\times d}$. 
\par
For a generic tensor from our orbit, the resulting system $Mx=B$ has unique solution $x\in\R^{d-s}$. 
If our system $Mx=B$ is not uniquely solvable,  we perform a randomized change of coordinates and repeat recursively. 
We show now that we can transform any tensor from the orbit ${\GL_{d}}*C$ so that the nonlinear relations for the lower Gau{\ss} operations from the previous section hold.

The proof of \Cref{conj:UpperTransformations} shows how we can replace the randomized change of coordinates in \Cref{alg:solve} by a loop over all cycles $(s,t)$ with $t>s$. Having mentioned this, we show that lower and diagonal transformations suffice to translate a generic tensor into a smaller orbit. 

\begin{theorem}\label{thm:linearSolSolvesAlgebraicSystem}
    For every $1\leq s<d-2$ and $G\in(\mathrm{I}_{s-1}\oplus\GL_{d-s+1})*C$, \Cref{alg:solve} returns $U^{(s,x)}W$ with $x\in\R^{d-s}$ and $W\in\GL_d$ such that $H:=U^{(s,x)}W*G$ satisfies the conditions \eqref{thm:LowerTransformations1} to \eqref{thm:LowerTransformations31} of \Cref{thm:LowerTransformations}. 
\end{theorem}

\begin{proof}For generic $G\in(\mathrm{I}_{s-1}\oplus\GL_{d-s+1})*C$, the system $Mx=B$ with  $M$ and $B$ according to \Cref{def:MB} is solvable according to \Cref{cor:MxB_solvable}. Furthermore, the matrix $M$ has full rank; see \Cref{prop:genericG_Mfullrank}. Therefore $x$ is unique and $H:=U^{(s,x)}*G$ satisfies the conditions \eqref{thm:LowerTransformations1} to \eqref{thm:LowerTransformations31} of \Cref{thm:LowerTransformations} thanks to \Cref{conj:UpperTransformations}. If the system $Mx=B$ is not uniquely solvable, we perform a change of coordinates $G':=W*G$ for a random $W$ and proceed recursively with $G'$. The resulting upper Gau{\ss} transformation has to be multiplied by $W$ from the right, since it acts on the new tensor in changed coordinates. 
\end{proof}

\section{The two- and three-dimensional case}\label{sec:lowdimCase}
It turns out that for the two- and three-dimensional case (that is, in the iteration step $s=d-1$ or $s=d-2$ of the recovery algorithm) the linear system $Mx=B$ in \Cref{alg:solve} is always underdetermined. However, in both cases, we can follow a different strategy for our upper Gau{\ss} transformations. We start with dimension two and construct explicit transformations that can be read from the tensor. 

\begin{proposition}\label{lem:transformationTwoDimensions}If $d=2$ and $G\in{\GL_2}*C$, then 
\begin{equation*}H:=
\begin{cases}

U^{(1,\frac{G_{2,1,1}-G_{1,2,1}}{G_{1,2,2}-G_{2,1,2}})}*G
&G_{1,2,2}\not=G_{2,1,2},\\
P^{(1,2)}*G&\text{otherwise,}
\end{cases}\end{equation*}
satisfies all the conditions of \Cref{thm:LowerTransformations}, that is  $H_{1,1,1}\not=0$, 
$H_{1,2,1}=H_{2,1,1}$, 
$H_{1,2,1}H_{2,1,1}= H_{1,1,1}H_{2,2,1}$ and 
$H_{2,1,1}H_{1,1,2} =H_{2,1,2}H_{1,1,1}$.  
\end{proposition}
\begin{proof}
Let  $G=A*C$ with $A\in\GL_2$. We show $A_{2,2}\not=0$ if and only if $G_{1,2,2}\not=G_{2,1,2}$.  
Consider the ideal  $I$ in the polynomial ring $R:=\R[a_{i,j},g_{i,j,k}\mid 1\leq i,j,k\leq 2]$ that is generated by the entries of $a*C-g$.
\par
First, we assume $A_{2,2}=0$.  
With Gr\"obner bases (see \cite{cox1997ideals,OSCAR-book} for an introduction) we can show $g_{1,2,2}-g_{2,1,2}\in I+\langle a_{2,2}\rangle$ and hence 
$G_{1,2,2}=G_{2,1,2}$.  
\par 
Conversely, assume $A_{2,2}\not=0$. With 
\begin{equation*}
3(a_{1,1}a_{2,2} - a_{2,1}a_{1,2})a_{2,2}- (g_{1,2,2}-g_{2,1,2})\in I
\end{equation*}
and  $A\in\GL_2$ we obtain 
\begin{equation*}
0\not=3\det(A)A_{2,2}=G_{1,2,2}-G_{2,1,2},
\end{equation*} thus $G_{1,2,2}\not=G_{2,1,2}$. In total we have showed
\begin{equation*}G_{1,2,2}=G_{2,1,2}\iff A_{2,2}=0.
\end{equation*}
Assuming $G_{1,2,2}=G_{2,1,2}$, we verify that 
$H=P^{(1,2)}*G$ satisfies the claimed relations from the proposition. We have 
\begin{equation*}
\left[
    \begin{array}{ccccc}
    H_{1,1,1} & H_{1,2,1} &\vrule & H_{1,1,2} & H_{1,2,2} \\
    H_{2,1,1} & H_{2,2,1} &\vrule & H_{2,1,2} & H_{2,2,2} 
    \end{array}
    \right]
    =
    \left[
    \begin{array}{ccccc}
    G_{2,2,2} & G_{2,1,2} &\vrule & G_{2,2,1} & G_{2,1,1} \\
    G_{1,2,2} & G_{1,1,2} &\vrule & G_{1,2,1} & G_{1,1,1} 
    \end{array}
    \right]
    \end{equation*}
    \noindent
so $H_{1,2,1}=H_{2,1,1}$ follows by assumption. Since $a_{2,2}=0$ and  \begin{equation*}
g_{2,2,2}-a_{2,1}^3\in I+\langle a_{2,2}\rangle
\end{equation*}
we use $A\in\GL_{2}$ and obtain 
\begin{equation*}
H_{1,1,1}=G_{2,2,2}=A_{2,1}^3\not=0.
\end{equation*}
Note that the relations $H_{1,2,1}H_{2,1,1}= H_{1,1,1}H_{2,2,1}$ and 
$H_{2,1,1}H_{1,1,2} =H_{2,1,2}H_{1,1,1}$ are analogous and can be found in the Github repository. 
\par
Assume now $G_{1,2,2}\not =G_{2,1,2}$, that is, $A_{2,2}\not=0$. 
We move on to the field of rational expressions $\R(a_{1,1},a_{1,2},a_{2,1},a_{2,2})$, set $g:=a*C$, and observe
\begin{equation*}x:=\frac{g_{2,1,1}-g_{1,2,1}}{g_{1,2,2}-g_{2,1,2}}=-\frac{a_{1,2}}{a_{2,2}}.
\end{equation*}
Hence, with $h:=U^{(1,x)}*g$ and a formal verification in \texttt{OSCAR} we see  
$h_{1,2,1} = h_{2,1,1}$, $h_{1,2,1}h_{2,1,1} = h_{1,1,1}h_{2,2,1}$ and $h_{2,1,1}h_{1,1,2} =h_{2,1,2}h_{1,1,1}$. The claim follows with an evaluation, given that $A_{2,2}\not=0$. \par
We assume $H_{1,1,1}=0$ for the sake of contradiction and consider the ideal $J$ over $\R[a_{i,j},h_{i,j,k}\mid 1\leq i,j,k\leq 2]$ generated by $a*C-h$ and the relations from the proposition. With \texttt{OSCAR}, $a_{1,2}h_{2,2,2}$ and $a_{1,1}h_{2,2,2}$ are in the radical of $J$, that is, $A_{1,2}H_{2,2,2}=A_{1,1}H_{2,2,2}=0$. As a consequence of shuffle relations, see \cite[Section 1.4]{bib:R1993} or \Cref{lem:diagonalElemCongruence}, we know $H_{2,2,2}\not=0$, so $A_{1,2}=A_{1,1}=0$, contradicting $A\in\GL_2$. 
\end{proof}

\begin{algorithm}
\begin{algorithmic}
\caption{Transformation matrix $\mathsf{up_2}(G)$}\label{alg:transformUp_s2}
\STATE\textbf{Input: } $G\in(\mathrm{I}_{d-2}\oplus\GL_{2})*C$\hfill\text{// $C$ according to \eqref{eq:coretensor}} 
   \IF{$G_{d-1, d, d}\not=G_{d, d-1, d}$}{
   \RETURN $U^{(d-1, \frac{G_{d,d-1,d-1}- G_{d-1,d,d-1}}{G_{d-1, d, d}-G_{d, d-1, d}})}$\;
   }
    \ELSE{
       \RETURN $P^{(d-1,d)}$\;}
       \ENDIF
       \end{algorithmic}
\end{algorithm}

Embedding \Cref{lem:transformationTwoDimensions} in higher dimensions, we can formulate an analogous to \Cref{thm:linearSolSolvesAlgebraicSystem} that provides the upper Gau{\ss} operation method for the special iteration step $s=d-1$ of our recovery algorithm, \cref{alg:transformUp_s2}. 

\begin{corollary}\label{cor:linearSolSolvesAlgebraicSystemd2}
    For all $G\in(\mathrm{I}_{d-2}\oplus\GL_{2})*C$, \cref{alg:transformUp_s2} returns $Q\in\GL_d$ such that $H:=Q*G$ satisfies the conditions of \Cref{thm:LowerTransformations} \eqref{thm:LowerTransformations1} to \eqref{thm:LowerTransformations31} with $s=d-1$. 
\end{corollary}

Similarly, we can embed the three-dimensional case for step $s=d-2$ and compute our upper Gau{\ss} transformations.  Before stating the algorithm, we continue with our running example.

\begin{example}\label{ex:d4recovery_precont}
We proceed with our path recovery from Example \ref{ex:d4recovery} for iteration step $s=2$. 
Our goal is to transform $G\in\mathrm{I}_1\oplus\GL_3$ from Example \ref{ex:tensor_running_example_start_s3} with upper Gau{\ss} operations so that we can apply \Cref{thm:LowerTransformations} for $s=2$. 
We observe that $G_{3,4,4}\not=G_{4,3,4}$. With the transformation \begin{equation*}
Q_1:=\mathrm{I}_4-\frac{G_{4,3,3}-G_{3,4,3}}{G_{3,4,4}-G_{4,3,4}}\mathrm{E}_{3,4}
\end{equation*} we obtain $G':=Q_1*G$ with the matrix folding
    \begin{equation*}\left[
    \begin{array}{ccccccccccccccccccc}1 & 0 & 0 & 0 & \vrule & 6 & 12 & -6 & 9 & \vrule & -3 & -6 & 3 & -3 & \vrule & 6 & 15 & -9 & 12 \\0 & 0 & 0 & 0 & \vrule & 0 & 8 & -6 & 10 & \vrule & 0 & -3 & 3 & -3 & \vrule & 0 & 10 & -9 & 13 \\0 & 0 & 0 & 0 & \vrule & 0 & -3 & 3 & -6 & \vrule & 0 & 0 & -1 & 1 & \vrule & 0 & -3 & 4 & -7 \\0 & 0 & 0 & 0 & \vrule & 0 & 4 & -3 & 7 & \vrule & 0 & 0 & 1 & -1 & \vrule & 0 & 4 & -4 & 8\end{array}
    \right]
    \end{equation*}
\noindent
and $G'_{3,4,4}\not=G'_{4,3,4}$. We continue with 
\begin{equation*}
Q_2:=\mathrm{I}_4-\frac{G'_{4,3,2}-G'_{3,4,2}}{G'_{3,4,4}-G'_{4,3,4}}\mathrm{E}_{2,4}
\end{equation*}
and obtain $G'':=Q_2*G'$ with folding
    \begin{equation*}\left[
    \begin{array}{ccccccccccccccccccc}1 & 0 & 0 & 0 & \vrule & 0 & 0 & 3 & -3 & \vrule & -3 & -3 & 3 & -3 & \vrule & 6 & 3 & -9 & 12 \\0 & 0 & 0 & 0 & \vrule & 0 & 0 & 2 & -2 & \vrule & 0 & -1 & 2 & -2 & \vrule & 0 & 1 & -5 & 5 \\0 & 0 & 0 & 0 & \vrule & 0 & -1 & -1 & 1 & \vrule & 0 & -1 & -1 & 1 & \vrule & 0 & 4 & 4 & -7 \\0 & 0 & 0 & 0 & \vrule & 0 & 1 & 1 & -1 & \vrule & 0 & 1 & 1 & -1 & \vrule & 0 & -4 & -4 & 8\end{array}
    \right]
    \end{equation*}
\noindent
and $G''_{2,3,3}\not=G''_{3,2,2}$. With the transformation 
\begin{equation*}Q_3:=\mathrm{I}_4-\frac{G''_{3,2,2}-G''_{2,3,2}}
{G''_{2,3,3}-G''_{3,2,3}}\mathrm{E}_{2,3}
\end{equation*}
the resulting tensor $H:=Q_3*G'$ has entries
    \begin{equation*}\left[
    \begin{array}{ccccccccccccccccccc}
    \mathbf1 & 0 & 0 & 0 & \vrule & 3 & 3 & 0 & 0 & \vrule & -3 & -6 & 3 & -3 & \vrule & 6 & 12 & -9 & 12 \\
    \mathbf0 & \mathbf0 & \mathbf0 & \mathbf0 & \vrule & \mathbf0 & \mathbf1 & 0 & 0 & \vrule & \mathbf0 & -3 & 3 & -3 & \vrule & \mathbf0 & 6 & -9 & 12 \\
    \mathbf0 & \mathbf0 & \mathbf0 & \mathbf0 & \vrule & \mathbf0 & \mathbf0 & \mathbf0 & \mathbf0 & \vrule & \mathbf0 & \mathbf0 & -1 & 1 & \vrule & \mathbf0 & \mathbf0 & 4 & -7 \\
    \mathbf0 & \mathbf0 & \mathbf0 & \mathbf0 & \vrule & \mathbf0 & \mathbf0 & \mathbf0 & \mathbf0 & \vrule & \mathbf0 & \mathbf0 & 1 & -1 & \vrule & \mathbf0 & \mathbf0 & -4 & 8\end{array}
    \right]
    \end{equation*}
\noindent
and clearly satisfies the conditions of \Cref{thm:LowerTransformations} for $s=2$. In fact, the lower and diagonal transformations are trivial here, that is, we can immediately continue with \Cref{thm:substab} for $s=2$ and conclude $H\in(\mathrm{I}_2\oplus\GL_2)*C$. Similarly as in Example \ref{ex:tensor_running_example_start_s3} we marked all entries in bold that are used for the theorem.  We proceed with iteration step $s=3$ in Example \ref{ex:d4recovery_cont}. 
\end{example}
The upper transformations of the example can be read from the tensor. In the case where a division of zero would occur, we instead perform a change of coordinates and proceed recursively. This yields \Cref{alg:transformUp_s3}. 

\begin{algorithm}[ht]
\begin{algorithmic}
\caption{Transformation matrix $\mathsf{up_3}(G)$}\label{alg:transformUp_s3}
\STATE\textbf{Input: }$G\in(\mathrm{I}_{d-3}\oplus\GL_{3})*C$ \hfill\text{// $C$ according to \eqref{eq:coretensor}}
\STATE $G_{\text{init}}\gets G$\;
   \IF{$G_{d-1, d, d}\not=G_{d, d-1, d}$}
   {
   \STATE $Q_1\gets \mathrm{I}_d+\frac{G_{d,d-1,d-1}- G_{d-1,d,d-1}}{G_{d-1, d, d}-G_{d, d-1, d}}\mathrm{E}_{d-1,d}$ \;
        \STATE$G\gets Q_1* G$\;
        \STATE
        $Q_2\gets \mathrm{I}_d+\frac{G_{d,d-1,d-2}- G_{d-1,d,d-2}}{G_{d-1, d, d}-G_{d, d-1, d}}\mathrm{E}_{d-2,d}$ \;
         \STATE$G\gets Q_2* G$\;
        \IF{$G_{d-2, d-1, d-1}\not=G_{d-2, d-1, d-1}$}{
        \STATE$Q_3\gets \mathrm{I}_d+\frac{G_{d-1,d-2,d-2}- G_{d-2,d-1,d-2}}{G_{d-2, d-1, d-1}-G_{d-2, d-1, d-1}}\mathrm{E}_{d-2,d-1}$ \;
       \STATE $G\gets Q_3*G$\;
        \IF{$G_{d-2, d-2, d-2}\not=0$}{
        \RETURN $Q_3Q_2Q_1$\;}
        \ENDIF
        }\ENDIF}
        \ENDIF
         \RETURN $\mathsf{up_3}(W*G_{\text{init}})W\text{ with random }W\in\mathrm{I}_{d-3}\oplus\GL_{3}$
         \end{algorithmic}
\end{algorithm}

\begin{proposition}\label{prop:proofCprects_3}
   For all $G\in(\mathrm{I}_{d-3}\oplus\GL_{3})*C$, \Cref{alg:transformUp_s3} returns $Q\in\GL_d$ such that $H:=Q*G$ satisfies the conditions \eqref{thm:LowerTransformations1} to \eqref{thm:LowerTransformations31} of \Cref{thm:LowerTransformations} for $s=d-2$. 
\end{proposition}

\begin{proof}
Similarly as in \Cref{lem:transformationTwoDimensions}, we show the statement for $d=3$ and embed it in any dimension $d\geq 3$. For a generic $G\in\GL_3*C$ the three transformations of the algorithm are sufficient. This follows from a formal verification in \texttt{OSCAR}. For this consider the 
field of rational expressions $K:=\R(a_{i,j}\mid 1\leq i,j\leq d)$ and  $g:=a*C$. 
With this we compute
    \begin{equation*}
    g':=\left(\mathrm{I}_3-\frac{g_{3,2,2}-g_{2,3,2}}{g_{2,3,3}-g_{3,2,3}}\mathrm{E}_{2,3}\right)*g,
    \end{equation*}
    then 
    \begin{equation*}g'':=\left(\mathrm{I}_3-\frac{g'_{3,2,1}-g'_{2,3,1}}{g'_{2,3,3}-g'_{3,2,3}}\mathrm{E}_{1,3}\right)*g',\end{equation*} 
    and finally 
    \begin{equation*}h:=\left(\mathrm{I}_3-\frac{g''_{2,1,1}-g''_{1,2,1}}{g''_{1,2,2}-g''_{2,1,2}}\mathrm{E}_{1,3}\right)*g''
    \end{equation*}
and verify $h_{i,1,1}=h_{1,i,1}$,
    $h_{i,1,1}h_{1,j,1}=h_{i,j,1}h_{1,1,1}$ and 
     $h_{i,1,1}h_{1,1,k}=h_{i,1,k}h_{1,1,1}$ in $K$, with indices according to the theorem.  
    \par 
    The expressions $g_{2,3,3}-g_{3,2,3}$, $g'_{2,3,3}-g'_{3,2,3}$, $g''_{1,2,2}-g''_{2,1,2}$ and $h_{1,1,1}$ are not constant zero; see for instance Example \ref{ex:d4recovery_precont}. Thus, with \cite[Lemma 1]{okamoto1973distinctness}, they do not evaluate to zero in expectation.  If they do, we perform a random change of coordinates and proceed recursively. 
\end{proof}

We have now discussed each step of our recovery algorithm in detail. We close this section by completing the specific path recovery from our running example. 
\begin{example}\label{ex:d4recovery_cont}We consider the path recovery task from Example \ref{ex:d4recovery} at iteration step $s=3$. With $H$ from Example \ref{ex:d4recovery_precont} we use the lower transformation $L^{(3,1)}$ and scaling $D^{(3,-1)}$. The result is our core tensor $C=D^{(3,-1)}L^{(3,1)}*H$. Keeping track of all Gau{\ss} operations $Q$ form Examples \ref{ex:tensor_running_example_start_s3}, \ref{ex:s1_fromH_to_s2}, \ref{ex:d4recovery_cont} and \ref{ex:d4recovery_precont}, we obtain 
 \begin{equation*}Q^{-1}=A=\begin{bmatrix}0 & -1 & 0 & -1 \\-1 & 1 & 0 & 1 \\1 & 0 & 0 & 1 \\-1 & 0 & 1 & 1\end{bmatrix}\in\GL_4.
 \end{equation*}
\end{example}

\section{The path recovery algorithm}\label{sec:proofs}

In this section, we formulate \Cref{alg:linear_learning}, which solves our path recovery problem (\Cref{thm:mainResult}). 

\begin{algorithm}
\caption{Path recovery}\label{alg:linear_learning}
\begin{algorithmic}
\STATE\textbf{Input:} $G\in{\GL_d}* C$\hfill\text{// $C$ according to \eqref{eq:coretensor}}
\STATE$Q\gets \mathrm{I}_d$
\FOR{$s =1,\dots,d-1$}
  \STATE $Q_{\mathsf{up}}\gets 
\begin{cases}
\mathsf{up_{\geq 4}}(G,s)&s\leq d-3\\
\mathsf{up_3}(G)&s=d-2\\
\mathsf{up_2}(G)&s=d-1. 
\end{cases}$\hfill\text{// output $Q_{\mathsf{up}}$ of \Cref{alg:solve}, \ref{alg:transformUp_s2} or \ref{alg:transformUp_s3}}
\STATE $H\gets Q_{\mathsf{up}}*G$
\STATE $y\gets-H_{s,s,s}^{-1}\begin{bmatrix}{H_{s,s+1,s}}&\dots& {H_{s,d,s}}\end{bmatrix}$\hfill\text{// $y\in\R^{d-s}$}
   \STATE $Q_{\mathsf{dlow}}\gets D^{(s,H_{s,s,s}^{-1})}L^{(s,y)}$\hfill\text{// $L$ and $D$ according to \eqref{eq:def_Lsy} and \eqref{eq:def_Dsy}}
    \STATE$G\gets Q_{\mathsf{dlow}}*H$
    \STATE$Q\gets Q_{\mathsf{dlow}}Q_{\mathsf{up}}Q$
\ENDFOR
\RETURN  $Q^{-1}D^{(s,G_{d,d,d}^{-1})}$\hfill\text{// returns $A$ such that $G=A*C$}
\end{algorithmic}
\end{algorithm}

We prove that \Cref{alg:linear_learning} is sound, i.e., we show \Cref{thm:mainResult}.

\begin{proof}[Proof of \Cref{thm:mainResult}]
From \cite[Corollary 6.3]{bib:PSS2019} we know that for any $G\in {\GL_d}*C$ there exists a unique $A\in\GL_d$ such that $G=A*C$. 
\par
We show that 
\Cref{alg:linear_learning} computes from $G$ a sequence  of transformations $Q^{(0)},\dots,Q^{(d)}\in\GL_d$ such that 
$Q:=Q^{(d)}\dots Q^{(0)}$
transforms $G$ to the core tensor $C$, that is 
\begin{equation}\label{eq:mainProofEq1}
Q*G=C,
\end{equation}
so we can use \cite[Theorem 6.2]{bib:PSS2019}, conclude $QA=\mathrm{I}_d$ and return $Q^{-1}=A$ as desired.  
\par
It remains to construct $Q^{(s)}$ for $0\leq s\leq d$ according to the algorithm and to show \eqref{eq:mainProofEq1}. 
We start with the initial $Q^{(0)}:=\mathrm{I}_d$ and $G^{(0)}:=G$.
We enter the loop of the algorithm, i.e., let  $1\leq s \leq d-1$ and assume recursively that we have constructed 
$Q^{(0)},\dots ,Q^{(s-1)}$ and 
\begin{equation}\label{eq:mainProofEq2}
G^{(s-1)}\in\left(\mathrm{I}_{s-1}\oplus\GL_{d-s+1}\right)*C.
\end{equation}
Note that for $s=1$ the condition \eqref{eq:mainProofEq2} clearly holds. 
According to \Cref{thm:linearSolSolvesAlgebraicSystem}, we 
 obtain  $x^{(s)}\in \mathbb{R}^{d-s}$ and $W^{(s)}\in\mathrm{I}_{s-1}\oplus\GL_{d-s+1}$, when using \Cref{alg:solve}, such that 
\begin{equation*}
H^{(s)}:=U^{(s,x^{(s)})}W^{(s)}*G^{(s-1)}
\end{equation*}
fulfills the conditions \eqref{thm:LowerTransformations1} to \eqref{thm:LowerTransformations31} of \Cref{thm:LowerTransformations}. 
Therefore $y^{(s)}\in\R^{d-s}$ 
with 
$y^{(s)}_i:=-H^{(s)}_{s,s+1,s}/H^{(s)}_{s,s,s}$ provides 
\begin{equation}\label{eq:mainProofEq3}
G^{(s)}:=D^{(s,H^{(s)}_{s,s,s})}L^{(s,y^{(s)})}*H^{(s)}\in\left(\mathrm{I}_{s}\oplus\GL_{d-s}\right)*C.
\end{equation}
Putting it all together, we get 
$Q^{(s)}:=D^{(s,H^{(s)}_{s,s,s})}L^{(s,y^{(s)})}U^{(s,x^{(s)})}W^{(s)}$ and obtain \eqref{eq:mainProofEq3} from \eqref{eq:mainProofEq2} via elementary Gau{\ss} operations. We proceed recursively until we reach $s=d-1$ with 
\begin{equation*}G^{(d-1)}\in\left(\mathrm{I}_{d-1}\oplus\GL_{1}\right)*C.
\end{equation*}
A last scaling $Q^{(d)}:=D^{(d,G^{(d-1)}_{d,d,d})}$ yields \eqref{eq:mainProofEq1} as desired. 
Within the loop of the recovery algorithm, the most computationally expensive step is solving the linear system. \Cref{cor:MxB_solvable} shows that this is generally possible, and it needs $\mathcal{O}(d^3)$ elementary operations. 
\end{proof}

\subsection*{Data availability}
The implementation of our proposed methods, along with the data and experiments, is publicly available at
\begin{center}\url{https://github.com/leonardSchmitz/efficient-path-recovery-from-signatures}.
\end{center}

\subsection*{Declaration of competing interest}
The author declares that he has no known competing financial interests or personal
relationships that could have appeared to influence the work reported in this paper.

\subsection*{Acknowledgments}
We thank Carlos Am\'endola for his help with the proof of \Cref{prop:genericG_Mfullrank}.

\bibliographystyle{alpha} 
\bibliography{refs.bib}

\end{document}